\documentclass[12pt,reqno]{amsart}

\usepackage{amssymb}

\newtheorem{theorem}{Theorem}[section]
\newtheorem{proposition}[theorem]{Proposition}
\newtheorem{lemma}[theorem]{Lemma}
\newtheorem{corollary}[theorem]{Corollary}

\theoremstyle{definition}
\newtheorem{definition}[theorem]{Definition}

\numberwithin{equation}{section}

\newcommand\bb[1]{{\text{\bf#1}}}

\newcommand\bbz{\mathbb{Z}} 

\newcommand\bbr{\mathbb{R}} 
\newcommand\bbc{\mathbb{C}}

\newcommand\bbh{\mathbb{H}}

\newcommand\bbi{\bb{I}}

\begin{document}
\baselineskip=15pt

\title[Pseudo-real principal $G$-bundles over a real curve]{Pseudo-real principal 
$G$-bundles over a real curve}

\author[I. Biswas]{Indranil Biswas}

\address{School of Mathematics, Tata Institute of Fundamental
Research, Homi Bhabha Road, Bombay 400005, India}

\email{indranil@math.tifr.res.in}

\author[O. Garc\'{\i}a-Prada]{Oscar Garc\'{\i}a-Prada}

\address{Instituto de Ciencias Matem\'aticas, C/ Nicol\'as
Cabrera, no. 13--15, Campus Cantoblanco, 28049 Madrid, Spain}

\email{oscar.garcia-prada@icmat.es}

\author[J. Hurtubise]{Jacques Hurtubise}

\address{Department of Mathematics, McGill University, Burnside
Hall, 805 Sherbrooke St. W., Montreal, Que. H3A 2K6, Canada}

\email{jacques.hurtubise@mcgill.ca}



\date{}

\begin{abstract}
We consider stable and semistable principal bundles over a smooth projective real 
algebraic curve, equipped with a real or pseudo-real structure in the sense of 
Atiyah. After fixing appropriate topological invariants, one can build a suitable 
gauge theory, and show that the resulting moduli spaces of pseudo-real bundles are 
connected. This in turn allows one to describe the various fixed point varieties on 
the complex moduli spaces under the action of the real involutions on the curve and 
the structure group.
\end{abstract}

\maketitle

\section{Introduction}\label{sec1}

The moduli spaces of vector bundles, and more generally principal bundles, on algebraic 
curves are some of the most studied spaces in geometry. Their strong ties to physics, 
through gauge theory, and their intricate structure has motivated much work over the 
last fifty years. To cite only some, one has the pioneering work of Narasimhan and 
Seshadri \cite{NS} relating the moduli of vector bundles to representations of the 
fundamental group into the unitary group, re-proven in a gauge theoretic context by 
Donaldson \cite{Do}; the foundational work of Atiyah and Bott \cite{AtBo}, placing the 
moduli in a gauge theoretic context, giving a basis for detailed Morse theoretic 
calculations; the subsequent explorations of the ring structure of the cohomology of 
the moduli by Jeffrey and Kirwan \cite{JK}. In parallel, Ramanathan \cite{Ra} developed 
a suitable theory for $G$-bundles, which again was put into a gauge theoretic context 
by Ramanathan and Subramanian \cite{RS}.

This study has mostly focused on complex 
curves but has recently begun to be extended to the case of real bundles over real 
curves (in the sense of Atiyah \cite{At}), that is, to the case when one has an 
antiholomorphic involutions on the curve and the group, and one is looking at bundles
with antiholomorphic involution compatible with these two involutions. 
This leads to the moduli spaces of semistable vector
bundles over a smooth projective real algebraic curves.
An examination of the gauge theoretic aspects of real vector 
bundles was considered in Biswas, Huisman and Hurtubise \cite{BHH}, and exploited by 
Liu and Schaffhauser \cite {LS} to compute mod $2$ Betti numbers of the spaces; see also 
Baird \cite{Ba}.

These moduli spaces of real vector bundles sit naturally inside the corresponding complex moduli, and indeed this is 
an advantage when trying to understand their gauge theory. In particular, the notions of 
(semi)stability for the two cases are related, and this allows one to exploit the gauge theory 
used in the complex case to understand the real case. This ambient picture for principal 
bundles was considered over curves in \cite{BH}, and more generally for K\"ahler manifolds 
in \cite{BGH}. In special cases, the fixed point sets of various involutions turn out to be related to the string theorists'
branes; see \cite{BS1}, \cite{BS2}, \cite{BG}. There is another approach, adopted in \cite{BHH} and then in \cite{LS} for 
vector bundles, which consists in building the involution into the gauge theory. This 
approach gives much better information about the topology. Of course, to do this, we must 
first classify the possible involutions. The answer turns out to be fairly elaborate. We 
begin by establishing some notation.

Let $G$ be a connected reductive affine algebraic group defined over $\mathbb C$. The center
of $G$ will be denoted by $Z$.
Let $$\sigma_G\, :\, G\, \longrightarrow\, G$$ be 
a real form on $G$, meaning an antiholomorphic involution. Let $G_\bbr\,=\,
G^{\sigma_{_G}}\,\subset\, G$ be the corresponding real group, the fixed point set of the involution. Fix a
maximal compact subgroup $K_G\, \subset\, G$
such that $\sigma_{_G}(K_G)\,=\, K_G$ (such a compact subgroup exists). Also, fix an element
$c$ in $Z_\bbr\,=\, Z\bigcap G_\bbr$. Let $X$ be an irreducible smooth complex projective
curve of genus $g(X)$, equipped with an anti-holomorphic involution $\sigma_X$.
The pair $(X\, , \sigma_X)$ is then a smooth projective {\it real} algebraic curve.

\begin{definition}Let $E$ be a holomorphic principal $G$--bundle over $X$. We will say
that $E$ is {\it pseudo-real}, or {\it $c$-pseudo-real} if there is an antiholomorphic
lift $\sigma_E$ to the total space of $E$ of the involution $\sigma_X$ on $X$:
$$
\begin{matrix}
E & \stackrel{\sigma_E}{\longrightarrow} & E\\
 \Big\downarrow && \Big\downarrow \\
X& \stackrel{\sigma_X}{\longrightarrow} & X
\end{matrix}
$$
which is compatible with the group action in the sense that 
$$\sigma_E(p\cdot g)\,=\, \sigma_E(p)\cdot \sigma_G(g)\, ,$$
and satisfies the condition
$$ \sigma_E^2 (p)\,=\, p\cdot c\, .$$
The pair $(E\, , \sigma_E)$ is called {\it real} if $c$ is the identity element $e$.
\end{definition}
 
We will first examine the topological classification of such bundles. This turns 
out to be quite intricate, even for a bundle over a point. Indeed, over a point, 
one ends up computing some Galois cohomology; in the course of writing this paper, 
after working out several examples, we discovered some work of J. Adams \cite{Ad}, 
who does these computations more systematically, and ties them to ``strong real 
forms'' of the group \cite{ABV}. Our explicit examples allow one to extend the 
results to classification of topological types over circles, and hence to real 
surfaces.

Passing then to the algebraic description of our bundles, we define stable,
semistable and polystable pseudo-real principal $G$--bundles on $X$, and show
the following:

\begin{theorem}
Let the genus $g(X)$ be at least two. For each topological type of pseudo-real 
bundle, there is a connected moduli space of semi-stable bundles on $X$.
\end{theorem}

\section{Topological classification of real curves}\label{se-top}

We recall the topological description of the possible real 
structures (i.e., anti-holomorphic involutions $\sigma_X$) on a Riemann surface
$X$ of genus $g$. More details can be found in \cite{BHH}.

These come in three types. For all three, one can write the surface $X$ 
as the union $X_0\bigcup \sigma_X(X_0)$ of two orientable surfaces with 
boundary, where the union is taken along the boundary. 

\begin{itemize} 
\item {\it Type 0}: This case is characterized by the fact that the real 
involution $\sigma_X$ has no fixed points. The quotient $X/\sigma_X$ is not 
orientable. In even genus, $X_0$ is obtained from a surface of genus 
$g/2$, by removing one disk; the boundary circle $\delta_1$ can be taken to be 
the concatenation of two intervals $I_0$ and $\sigma_X(I_0)$. In odd 
genus, the surface is obtained from a surface of genus $(g-1)/2$ by 
removing two disks. The boundary is then two circles, $\delta_1 
\delta_2\,=\, \sigma_X(\delta_1)$ interchanged by $\sigma_X$.

\item{\it Type I}: This case is characterized by the fact that the real 
involution $\sigma_X$ has $r>0$ fixed circles, and that the quotient 
$X/\sigma_X$ is orientable. The surface $X_0$ is simply the quotient 
$X/\sigma_X$ with boundary the $r$ fixed circles $\gamma_i$, $i\,=\,1,\cdots 
,r$.

\item{\it Type II}: This case is characterized by the fact that the real 
involution $\sigma_X$ has $r>0$ fixed circles, and that the quotient 
$X/\sigma_X$ is not orientable. The real involution $\sigma_X$ has $r$ fixed 
circles, and the quotient $X/\sigma_X$ is not orientable. The surface 
$X_0$ is of genus $(g -r-1)/2$, with $r+1$ disks removed. One of the boundaries $\delta_1$ 
can be written as the concatenation of two intervals $I_0$ and
$\sigma_X(I_0)$; the others are the $r$ fixed circles $\gamma_i$, $i 
\,=\,1,\cdots ,r$.
\end{itemize}
 
The surface $X_0$ has a standard decomposition into a union of cells, 
with all the $0$-cells on the boundary, $1$-cells which, apart from those in 
$\gamma_i, \delta_j $, have interiors lying in the interior of 
$X_0$ (these include the cells defining the standard first homology 
basis), and a single $2$-cell. The surface $X$ then has an induced 
decomposition for which, apart from the 0-cells and the 1-cells in 
$\gamma_i, \delta_i$, all cells come in pairs $c, \sigma_X(c)$.

\section{Classifying pseudo-real bundles}

\subsection{Normalizing $c$.}
The elements $a$ of the center $Z$ act on $E$ as automorphisms. If one modifies the map
$\sigma_E$ to $\sigma_E\cdot a$, one finds that the square $\sigma^2_E = c$ gets modified
to $c\sigma_G(a)a$. In particular, if $c\,=\, a^{-1}\sigma_G(a^{-1})$,
$a\, \in\, Z$, we can normalize $c$ to the identity $e$, i.e., make the structure real. We quotient out by this equivalence. Let us define 
\begin{equation}\label{2c}
H^2(\bbz/2\bbz,\, Z) \,=\, Z_\bbr /\{\sigma_G(a)a\,\mid\, a\,\in\, Z\}\, ;
\end{equation}
these classify the parameter $c$ in the definition of pseudo-real structure.
We note that all the elements of $Z_\bbr$ which
are squares map to zero in $H^2(\bbz/2\bbz,\, Z)$; all elements of
$H^2(\bbz/2\bbz,\, Z)$ are of
order two. Furthermore, we have a surjective map from the subgroup $Z_{\mathbb R}(2)
\, \subset\, Z_{\mathbb R}$ of order 2 points to 
 $Z_{\mathbb R}/\{a^2\,\mid\, a\,\in\, Z_{\mathbb R}\}$ and so to $H^2(\bbz/2\bbz,\, Z)$. Thus
we can suppose that all our elements of $H^2(\bbz/2\bbz,\, Z)$ are represented by elements of order 2, which, in particular will live in every maximal compact subgroup.
 
\subsection{Classifying bundles over a point}

Let us consider real $G$-bundles $E$ over a point. These are not the same as 
$G_\bbr$-bundles. The question hinges on whether the involution $\sigma_E$ on the 
principal bundle has a fixed point. Indeed, trivializing the bundle by
choosing a point $z_0$ of the bundle, let $\sigma_E(z_0) 
\,=\, z_0\cdot h$. One then has, by the defining property of real structures on bundles, that 
$$\sigma_E(z_0g)\,=\, z_0\cdot h\cdot \sigma_G(g)\, .$$ The involutive nature of $\sigma_E$ then
forces 
\begin{equation}\label{cocycle}
z_0\,=\, \sigma_E^2(z_0) \,=\,z_0\cdot h\cdot \sigma_G(h)\, .
\end{equation}
 
\begin{lemma}
Isomorphism classes of real $G$--bundles over a point are classified by the 
non-Abelian group cohomology $H^1(\bbz/2\bbz,\, G)$.
\end{lemma}
 
\begin{proof}
The cocycle condition for this cohomology is precisely \eqref{cocycle}, namely
$$
\{h\, \in\, G\, \mid\, \sigma_G(h)h\,=\, e\}\, .
$$
The coboundary equivalence, in turn, is given by $h\,\simeq\, \sigma_G(b)hb^{-1}$,
$b\, \in\, G$. But this corresponds to changing the base point $z_0$ to $z_0\cdot b$.
Therefore, the lemma follows.
\end{proof}
 
The equivalence class of $e$ makes the non-Abelian group cohomology $H^1(\bbz/2\bbz,\, G)$
a pointed set. The base point of $H^1(\bbz/2\bbz,\, G)$ corresponds to the trivial $G_\bbr$
bundle, in other words, it is the unique real $G$--bundle with a real point (a point
fixed by $\sigma_E$).

In a similar fashion to \eqref{cocycle}, for $c$-pseudo-real structures, one has shifted cocycles
\begin{equation}
z_0\cdot c\,=\, \sigma_E^2(z_0) \,=\,z_0\cdot h\cdot \sigma_G(h)\,=\,
z_0\cdot (\sigma_G(h)h)\, .\label{c-cocycle}
\end{equation}
Note that since $c$ is central, $\sigma_G(h)h\,=\, h\sigma_G(h).$ Quotienting the set of shifted cocycles out by a coboundary equivalence
$h\,\simeq\, b^{-1}h\sigma_G(b)$, $b\, \in\, G$, gives a set, the $c$-shifted non-Abelian
group cohomology $H^1_c(\bbz/2\bbz,\, G)$, and one obtains:
 
\begin{lemma}
Isomorphism classes of pseudo-real $G$-bundles over a point corresponding to an
element $c$ in $H^2(\bbz/2\bbz,\, Z)\,=\, Z_\bbr / \sigma_G(Z)\cdot Z$ are classified by the shifted non-Abelian
group cohomology $H^1_c(\bbz/2\bbz,\, G)$.
\end{lemma}

\begin{proposition}\label{prop2.1}
The set $H^1_c(\bbz/2\bbz,\, G)$ is discrete.
\end{proposition}

\begin{proof}
Let us look at the infinitesimal conditions, around a fixed element $h$ satisfying
\begin{equation}\label{cohcon}
c \,=\, h \sigma_G(h)\, .
\end{equation}
Consider $h(t) \,=\, h\exp(tv)$, where $v$ is a
function from a neighborhood of $0\, \in\, \mathbb R$ to the Lie algebra $\mathfrak g$
of $G$; the image of $0\, \in\, \mathbb R$ will be denoted by $v_0$. Impose the condition
that $h(t) \sigma_G(h(t))\,=\, c$, so
\begin{equation}\label{f1}
h\exp(tv)\sigma_G(h)\exp(td\sigma_G(v))\,=\, c\, ,
\end{equation}
where $d\sigma_G\,:\,
\mathfrak g\, \longrightarrow\, \mathfrak g$ is the homomorphism of Lie algebras
corresponding to $\sigma_G$. Since $\sigma_G(h)\,=\, h^{-1}c$, and $c$ is in the
center, from \eqref{f1} we have
$$
h\exp(tv)h^{-1}\exp(td\sigma_G(v))\,=\, e\, .
$$
Therefore, taking derivative at $t\,=\, 0$, we have
\begin{equation}\label{f2}
\text{Ad}(h^{-1})(v_0)+ d\sigma_G(v_0)\,=\, 0 
\end{equation}
(recall that $v_0\,=\, v(0)$).

Next consider the tangent space to the orbit $\{b^{-1}h\sigma_G(b)\}_{b\in G}$ at
$h$. Write $b(t)\,=\, \exp(tw)$ with $w$ is a
function from the neighborhood of $0\, \in\, \mathbb R$ to $\mathfrak g$.
Taking derivative of $\exp(-tw)h\exp(tw)$ at $t\,=\, 0$ we have
\begin{equation}\label{f3}
v_0\, :=\, \text{Ad}(h)(-w_0)+ d\sigma_G(w_0)\, ,
\end{equation}
where $w_0 \,=\, w(0)$. From \eqref{cohcon} we have
\begin{equation}\label{f4}
(d\sigma_G)\circ \text{Ad}(h) -\text{Ad}(h^{-1})\circ d\sigma_G\,=\, 0\, ,
\end{equation}
because $\sigma_G\circ\text{Ad}(h)(\exp(u))\,=\, h\sigma_G(\exp(u))h^{-1}\,=\,
\text{Ad}(h^{-1})\circ\sigma_G (\exp(u))$. Note that from \eqref{f4} it follows
immediately that $v_0$ in \eqref{f3} satisfies the equation in \eqref{f2}.

Consider the $\mathbb R$--linear operator
$$
T\, :\, {\mathfrak g}\, \longrightarrow\, {\mathfrak g}\, ,~ \ 
v\, \longmapsto\, \text{Ad}(h^{-1})(v)+ d\sigma_G(v)\, .
$$
Since $\text{kernel}(T)\bigcap \sqrt{-1}\cdot\text{kernel}(T)\,=\,0$, we have
\begin{equation}\label{f5}
\dim_{\mathbb R} \text{kernel}(T)\, \leq\, \dim_{\mathbb C} \mathfrak g\, ,
~\ \dim_{\mathbb R} \text{image}(T)\, \geq\, \dim_{\mathbb C} \mathfrak g\, .
\end{equation}
Now consider the $\mathbb R$--linear operator
$$
T'\, :\, {\mathfrak g}\, \longrightarrow\, {\mathfrak g}\, ,~ \ 
v\, \longmapsto\, -\text{Ad}(h)(v)+ d\sigma_G(v)\, .
$$
As before, $\text{kernel}(T')\bigcap \sqrt{-1}\cdot\text{kernel}(T')\,=\,0$, and
$$
\dim_{\mathbb R} \text{image}(T')\, \geq\, \dim_{\mathbb C} \mathfrak g\, .
$$
Combining this with \eqref{f5} and the above observation that
$\text{image}(T')\, \subset\, \text{kernel}(T)$, we conclude that
$\text{image}(T')\,=\, \text{kernel}(T)$. But
$\text{kernel}(T)/\text{image}(T')$ is the tangent space to $H^1_c(\bbz/2\bbz,\, G)$
at $h$. Therefore, the set $H^1_c(\bbz/2\bbz,\, G)$ is discrete.
\end{proof}

\begin{corollary}\label{cor1}
The set $H^1_c(\bbz/2\bbz,\, G)$ is finite.
\end{corollary}

\begin{proof}
Since the cocycle condition in \eqref{cohcon} is given by a real algebraic equation, it
follows that the subset of $G$ satisfying \eqref{cohcon} is a real algebraic variety.
In particular, it has finitely many connected components. In the proof of Proposition
\ref{prop2.1} we have seen that tangent space to $H^1_c(\bbz/2\bbz,\, G)$ is zero. From
this it follows that each equivalence class for the coboundary condition is a
connected component in the subset of $G$ satisfying \eqref{cohcon}. Therefore, from
Proposition \ref{prop2.1} it follows that $H^1_c(\bbz/2\bbz,\, G)$ is a finite set.
\end{proof}

We will now give an explicit description of the set $H^1_c(\bbz/2\bbz,\, G)$.

Let $\tau_G$ be a Cartan involution of $G$ defining a compact real form. The fixed point
set $G^{\tau_G}$, which is a compact group, will be denoted by $K$. Let $G\,=\, KM$ be
the Cartan decomposition, where
$$
M\,:=\, \{m\, \in\, G\, \mid\, \tau_G(m) \,= \, m^{-1}\}\, .
$$
We can suppose by a result of Cartan, that $\sigma_G\,=\, \theta \tau_G$, where
$\theta$ is a holomorphic involution of $G$. All of these involutions commute, and
so all the involutions map $K,M$ to themselves.

\begin{proposition}
Let $h$ represent a cohomology class in $H^1_c(\bbz/2\bbz,\, G)$. By the action 
$h\,\longmapsto \,\sigma_G(a)h a^{-1}$, one can normalize $h$ to an element $k$
lying in $K$.
 
One can normalize $k$ to an element of order $1\, ,2$ or $4$ ($1\, ,2$ if $c\,=\,1$)
lying in $K' \,= \, K^\theta$, which is defined up to conjugation in $K'$, and so
can be taken to lie in the set $T'_4$ of points of order $1,2$ or $4$ in a fixed
maximal torus $T'$ of $K'$. Two such elements $k'\, ,\widehat{k}'$ of $T'_4$ define the
same class in $H^1_c(\bbz/2\bbz, \, G)$ if there is an $a\,\in\, K$ with
$\theta (a) k' a^{-1} \,=\, \widehat{k}'$.
\end{proposition}
 
\begin{proof} 
We have $c\,=\,\sigma_G(h)h$ and so, writing $h \,= \,km, \, m\,\in\, M, \,
k\,\in\, K$, we have $c\,=\, \theta(k) k k^{-1}\theta(m^{-1})km$. Thus, $cm^{-1}\,=\, 
(\theta(k) k)( k^{-1}\theta(m^{-1})k)$; since the Cartan decomposition is unique and
$c$ lies in $K$, we have $c\,=\, \theta(k) k$, $1\,=\, 
k^{-1}\theta(m^{-1})km$, and so $1\,=\, k^{-1}\theta(m^{-1/2})km^{1/2}\,=\, 
k^{-1}\sigma_G(m^{1/2})km^{1/2}$. Setting $a\,=\, m^{1/2}$, we get $\sigma_G(a) h a^{-1}\,
=\,k k^{-1}\sigma_G(m^{1/2})km m^{-1/2}\,=\, k$, so that one can normalize $h$ to lie in $K$.
 
One has, since $c\,=\,\theta(k) k$, that $\theta(k)\, , k$ commute. Let $K_{k,\theta(k)}$ be 
the group of elements of $K$ which commute with $k\, , \theta(k)$; this is $\theta$ 
invariant, and let us take a maximal $\theta$ invariant torus $T$ inside this group, 
which will contain $k, \theta(k)$. Within $T$, let us choose a decomposition $k\,=\, 
k'k''$, with $\theta(k')\,=\, k'\, , \theta(k'')\,=\, (k'')^{-1}$; acting by a $(k'')^{1/2}$ 
within $T$ reduces $k$ to a $k'$ with $\theta(k')\,=\,k'$. The equation then becomes 
$(k')^2\,=\, c$, and so $(k')^4 \,=\, 1$. The set $K'_4$ of such $k'$ is then to be 
considered modulo conjugation in $K' \,=\, Fix(\theta)\,=\, K^\theta \,
\subset\, K$, and so can be taken to lie 
in the set $T'_4$ of elements of order at most $4$ in a fixed maximal torus $T'$ of $K'$. 
The coboundary condition defines equivalence classes by saying that two elements 
$k'\, ,\widehat{k}'$ of $T'_4$ are equivalent if there is an $a\,\in\, K$ with
$\theta(a) k' a^{-1}\,=\, \widehat{k}'$, in particular if they are in the same
orbit under the Weyl group $\text{Weyl}(K')$.
\end{proof}

The center of $G$ will be denoted by $Z$. Consider the exact sequence of groups
$$
1\,\longrightarrow\, Z\,\longrightarrow\, G\,\longrightarrow\,G_{ad}\,\longrightarrow\, 1\, ,
$$
where $G_{ad}$ is the adjoint group. As in \eqref{2c}, define 
 $$H^2(\bbz/2\bbz,\, Z)\,=\, Z_\bbr /\{\sigma_G(a)\cdot a\, \mid\, a\,\in\, Z\}\, .$$
As noted following \eqref{2c}, we can suppose that all our elements $c$ of
$H^2(\bbz/2\bbz,\, Z)$ are represented by elements of order two.

\begin{proposition}\label{sequence}
There is an exact sequence of pointed sets
$$
H^1(\bbz/2\bbz,\, Z)\,\longrightarrow\, H^1(\bbz/2\bbz,\, G)\,
\longrightarrow\, H^1(\bbz/2\bbz,\, G_{ad})\,\longrightarrow\, H^2(\bbz/2\bbz,\, Z)\, .
$$
Each element $h\, \in\, H^1(\bbz/2\bbz,\, G_{ad})$ lift to $H_c^1(\bbz/2\bbz,\, G)$ for
some $c$ in $H^2(\bbz/2\bbz,\, Z)$. Under the last homomorphism of the sequence, the
image of any $h\, \in\, H^1(\bbz/2\bbz,\, G_{ad})$ is this $c\,\in\,
H^2(\bbz/2\bbz,\, Z)$.
\end{proposition}

\begin{proof} 
Consider an $h\,\in\, G_{ad}$ defining a class in $H^1(\bbz/2\bbz,\, G_{ad})$. It satisfies the
condition $\sigma_G(h)h \,=\, e$. Its lift $\widetilde h$ in $G$ then satisfies the constraint 
$\sigma_G({\widetilde h}){\widetilde h} \,=\, c$, with $c$ a real element of the center; we 
note that the different choices involved tell us that in fact we have a class in 
$H^2(\bbz/2\bbz,\, Z)$, and we take this to be the coboundary. If this
class in $H^2(\bbz/2\bbz,\, Z)$ is trivial, then $h$ 
lifts to an element $\widetilde h$ of $H^1(\bbz/2\bbz,\, G) $; more generally, the lift is to 
an element of $H_c^1(\bbz/2\bbz,\, G)$. The ambiguity of the lift is by an element $a$ of 
the center, satisfying $\sigma_G(a)a \,= \,1$, and so the fiber in $H_c^1(\bbz/2\bbz,\, G)$ 
lying over $h$ is an ``orbit'' of $H^1(\bbz/2\bbz,\, Z)$; this orbit might not be free, 
however.
\end{proof}

\subsection{Real bundles and real forms of the group}

The classification of real forms for bundles mimics very closely the classification 
of real forms for the group. For the group, real forms $\sigma_G,\, \tau_G$ differ 
by an automorphism $\theta$ of the group: $\tau_G \,=\, \theta\sigma_G$. Indeed, let 
$\sigma^0_G$ denote an anti-holomorphic involution of $G$ giving the compact real 
form. The real forms of $G$ are obtained from $\sigma^0_G$ by composing with an 
involution of order two, which could be an inner, or outer automorphism. More 
generally, the different anti-holomorphic involutions are grouped into equivalence 
classes whereby two automorphisms are equivalent if they are related by an inner 
automorphism. Thus, for example, the real involutions for ${\rm SL}(n, {\mathbb 
C})$ which give the various ${\rm SU}(p,n-p)$, $0\,\leq\, p\, \leq\, n$, are all in 
the same equivalence class.

Now fix a $\sigma_G$ in such an inner equivalence class.

Two real forms $\tau_G$ and $\tau'_G$ of $G$ are 
called \textit{equivalent} if there is an element $g\, \in\, G$ such that 
$\tau'_G\,=\, \text{Ad}_g\circ\tau_G\circ (\text{Ad}_g)^{-1}$.

We have the following:

\begin{proposition}
The set of equivalence classes of real forms obtained from $\sigma_G$ by an inner 
automorphism is $H^1(\bbz/2\bbz,\, G_{ad})$: for $k\,\in\, G$ representing a class in 
$H^1(\bbz/2\bbz,\, G_{ad})$, we can define another real form by $\sigma_G^k \,=\, 
{\rm Ad}_k\circ \sigma_G$.
\end{proposition}

\begin{proof} Consider all $h\, \in\, G$ such that
$\text{Ad}_h\circ \sigma_G$ is an involution. This
gives the condition $h\sigma_G(h)$ central, and so $h\sigma_G(h) \,=\, e$ in $G_{ad}$ (inner automorphisms are given by the
adjoint group). One would want to consider
as equivalent real forms 
$$\text{Ad}_h\circ \sigma_G\, ,\, \ \
\text{Ad}_a\circ \text{Ad}_h\circ \sigma_G\circ (\text{Ad}_a)^{-1}$$
 giving an equivalence $h\,\simeq\, ah\sigma_G(a^{-1})$, and so our cohomology group.\end{proof}
 
 Now suppose that the $k\in G$ is such that $k\sigma_G(k) \,=\, c$. This gives us a 
class in $H^2(\bbz/2\bbz,\, Z)$.
 
 \begin{proposition}
The set of pseudo-real forms of bundles for the real structure $\sigma_G$, and 
class $c'\in H^2(\bbz/2\bbz,\, Z)$ and those for the real structure $\sigma_G^k$, 
and class $c'c^{-1}\,\in\, H^2(\bbz/2\bbz,\, Z)$ are in bijection: if $\sigma_E$ is 
a real bundle structure for the real structure $\sigma_G$, then setting
$$
\sigma^k_E(e\cdot g) \,\stackrel{def}{=}\,\sigma_E(e\cdot g)\cdot k^{-1}\, ,
$$
we have that $\sigma^k_E$ is a real bundle structure for the real structure
$\sigma^k_G$, giving a bijection
$$H^1_{\sigma_G}(\bbz/2\bbz,\, G)\,\longrightarrow\, H^1_{{\sigma}^k_G}(\bbz/
2\bbz, \,G)\, ,$$
where $H^1_{\sigma_G}(\bbz/2\bbz,\, G)$ and $H^1_{{\sigma}^k_G}(\bbz/
2\bbz, \,G)$ are $H^1(\bbz/2\bbz,\, G)$ for the involutions $\sigma_G$ and
$\sigma^k_G$ respectively.
\end{proposition}

\begin{proof}The second statement is first a matter of checking that the relation $\sigma^k_E(e\cdot g) = \sigma^k_E(e) \cdot \sigma^k(g)$ is satisfied; then one has that 
$$(\sigma^k_E)^2(e)\,=\, \sigma^k_E(\sigma_E(e)\cdot k^{-1})
\,= \,\sigma_E(\sigma_E(e)\cdot k^{-1})\cdot k^{-1}\,=\,
\sigma_E(\sigma_E(e))\cdot \sigma_G(k^{-1})\cdot k^{-1}\,=\, c'c^{-1}\, .$$
This completes the proof.
\end{proof}

Thus, for inner equivalent real structures on the group, we have the cohomology $H^1(\bbz/2\bbz,\, G_{ad})$; on the other hand, the set of real structures for a principal $G$--bundle over a point is
$H^1(\bbz/2\bbz,\, G)$, with the pseudo-real structures for $c$ being given by $H_c^1(\bbz/2\bbz,\, G)$. The two concepts are very close, with the obvious remark that for adjoint
groups they coincide. For more general reductive groups, to effect the classification of real and pseudo-real bundles for the various real 
forms of the group, it will suffice to classify for one real form for the group in each inner equivalence class. Moreover, the preimage in $H^1(\bbz/2\bbz,\, G)$ (see \ref {sequence}) of an element in $H^1(\bbz/2\bbz,\, G_{ad})$ is a copy of $H^1(\bbz/2\bbz,\, Z)$; this is not automatic in an exact sequence of pointed sets, but follows here from the homogeneity described above, and the fact that $H^1(\bbz/2\bbz,\, Z)$ is the same for all inner equivalent real forms of the group.

The cohomologies $H^1(\bbz/2\bbz,\, G)$ are computed in \cite[Section~9]{Ad} for a 
variety of groups, both classical and exceptional; we give here their explicit 
representatives, for a variety of examples, as these are useful in understanding 
reality conditions over a circle, which is the next step.

\subsection{Examples}

{\it (a)\, $G\,=\,\bbc^*$, with the automorphism $\sigma_G (a)\,=\, \overline a^{-1}$}. One 
computes $H^2(\bbz/2\bbz,\, \bbc^*)\,=\,\{1\}$. The set of cocycles is then $\bbr^*$, and 
the coboundaries $\bbr^+$. The cohomology classes in $H^1(\bbz/2\bbz, \,\bbc^*)\,=\, 
\bbz/2\bbz$ are represented by $\pm 1$.

{\it (b)\, $G\,=\,\bbc^*$, with the automorphism $\sigma_G (a)\,=\, \overline a$}. One has
$H^2(\bbz/2\bbz,\, \bbc^*)\,= \,\{\pm 1\}$. The set of cocycles and of coboundaries are
both the unit circle; there is a single real structure. On the other hand, for $c\,=\,
-1$, there are no cocycles, as one has to solve $\overline{h} h\,=\, -1$. One has $H^1(\bbz/2\bbz,
\, \bbc^*)\,=\, \langle 1\rangle$.

{\it (c)\, ${\rm GL}(n,{\mathbb C})$, $n>2$, with $\sigma_G(g)\,=\, (g^*)^{-1}$, so that
the fixed subgroup is ${\rm U}(n)$}. In this case, one has $(h^*)^{-1}h \,=\, \bbi$.
Changing trivializations modifies $h$ to $a^*ha$. Taking the polar decomposition $h\,=\,
u\cdot p$ ($u$ unitary and $p$ hermitian positive) of $h$, gives $up^{-1}up\,=\,\bbi$, or
$up^{-1}u\,=\,p^{-1}$. The unitary matrices act on $h$ by conjugation, and so also on
$u$, $p$. We can by a unitary change of trivialization diagonalize $u$ to
$diag(\exp(\sqrt{-1}\theta_j))$; one then has for the diagonal entries of $p^{-1}$, that $p^{-1}_{ii}
\,=\,\exp(2\sqrt{-1}\theta_j) p^{-1}_{ii}$, and so, by positivity, that the eigenvalues of $u$
are $\pm 1$. This then tells us that $u^{-1}\,=\,u$, and so, up to conjugation
$diag(1,\cdots ,1,-1,\cdots ,-1)$. The relation $u^{-1}p^{-1}u\,=\,p^{-1}$ tells us then that $p$ is
also block diagonal, and indeed one can further normalize so that it is diagonal. Finally,
acting by $a$ a positive diagonal real matrix, one can normalize $p$ to $\bbi$, and so $h$
to $diag(1,\cdots,1,-1,\cdots ,-1)$. This then leaves one invariant, the signature, and so 
$H^1(\bbz/2\bbz,\, {\rm GL}(n,{\mathbb C}))\,=\, \{0,1,\cdots ,n\}$ is of cardinality $n+1$.

{\it (d)\, ${\rm PGL}(n,{\mathbb C})$, $n>2$, with $\sigma_G(g)\,=\,(g^*)^{-1}$, so that
the
fixed subgroup is ${\rm PU}(n)$}. One proceeds as above, except that now $diag(1,\cdots ,1,
-1,\cdots,-1)$ is equivalent to $$diag(-1,\cdots ,-1,1,\cdots ,1)\, ;$$ then $H^1(\bbz/2\bbz,
\, {\rm PGL}(n,{\mathbb C}))\,=\, \{0,1,\cdots ,n\}/\bbz/2\bbz$, where the action is by
$k\,\longmapsto\, n-k$. Combining with the results of examples (a) and (c), the
sequence in Proposition \ref{sequence} becomes 
$$ \{\pm1\}\,\longrightarrow\, \{0,1,\cdots ,n\}\,\longrightarrow\,
\{0,1,\cdots ,n\}/\bbz/2\bbz\,\longrightarrow\, \{\pm1\}\, .$$

{\it (e)\, ${\rm GL}(n,{\mathbb C})$, $n>2$, with $\sigma_G(g)\,=\,\overline g$}: in this
case, the cocycle $\widetilde\sigma(e)\,=\,h$, gives a real endomorphism of $\bbc^n$ by
$T(a)\,=\,h(\overline a)$. This is anti-linear, so that if $I$ is multiplication by
$\sqrt{-1}$, one has $TI\,=\,-IT$. Also, $T$ has square the identity, and so has
$\pm 1$ eigenspaces, both of real dimension $n$ (as they are interchanged by multiplication
by $\sqrt{-1}$), and both spanning $\bbc^n$ as complex vector spaces. The coboundary
equivalence $h\,\longmapsto\,\overline g hg^{-1}$ gives $T\,\longmapsto\,{\overline g} T
{\overline g}^{-1}$, and so one can in essence change bases so that the $+1$ eigenspace of
$T$ corresponds to the standard basis, normalizing $h$ to the identity. This gives
$H^1(\bbz/2\bbz,\, {\rm GL}(n,{\mathbb C})) \,=\, \{ 1\}$.

In the same way, for $H^1_{-\bbi}(\bbz/2\bbz,\, {\rm GL}(n,{\mathbb C}))$, one builds an
anti-linear $T$ with $T^2\,=\, -\bbi$; this is impossible in odd dimension, and in even
dimension, one can normalize $h$ to 
$$J \,= \,\begin{pmatrix} 0&-\bbi\\ \bbi&0\end{pmatrix}$$

{\it (f)\, ${\rm PGL}(n,{\mathbb C})$, $n>2$, with $\sigma_G(g)\,=\, {\overline g}$}. This
basically repeats the two calculations for ${\rm GL}(n,{\mathbb C})$, and so 
$$H^1(\bbz/2\bbz,\, {\rm PGL}(n,{\mathbb C}))\,=\, \{ 1\}$$ for $n$ odd, and
$H^1(\bbz/2\bbz,\, {\rm PGL}(n,{\mathbb C}))\,=\,\{ \pm 1\}$ for $n$ even. Combining
with the results of examples (b) and (d), the sequence
in Proposition \ref{sequence} becomes 
$$0\,\longrightarrow\, 0\,\longrightarrow\, 0\,\longrightarrow\, \{ \pm 1\}$$
for $n$ odd, and 
$$0\,\longrightarrow\, 0\,\longrightarrow\,\{ \pm 1\}\,\longrightarrow\,\{ \pm 1\}$$
for $n$ even.

{\it (g)\, ${\rm SO}(2n,\bbc), n>1$, with $\sigma_G(g)\,=\, \overline g$}. The elements $g$
of this group satisfy $g^T \,=\,g^{-1}$. Here the real group is the compact group
${\rm SO}(2n)$. The center is $\{\pm 1\}$. The cocycle condition $\overline h h \,=\, 1$
tells us that $h$ is hermitian, as well as being orthogonal. One can diagonalize such an
$h$ with a unitary matrix. The result must have eigenvalues that are real, as well as being
of norm one, so the result is then a matrix $D_{2k}$ with $2k$ eigenvalues that are $1$ and
$2n-2k$ that are $-1$. This tells us that $h \,=\, uDu^{-1}$. Any two such $u$ differ by
an element of the stabilizer of $D$, and so we can normalize $u$ to a unique form 
$$\begin{pmatrix} \bbi& a\\ b& \bbi\end{pmatrix}\, .$$
Then also $h\,=\, \overline h^{-1}$ tells us that $h\,=\,\overline uD\overline u^{-1}$, so
that $u$ and $\overline u$ coincide. The element $u$ is then an orthogonal matrix. The cohomology
classes are then the possible matrices $D$. The same argument works for ${\rm SO}(2n+1,\bbc)$.

For the $c\,=\, -1$, one has $h$ normalizable to $J$, as for ${\rm GL}(n,{\mathbb C})$. 

{\it (h)\, ${\rm SO}(2n,\bbc), n>1$, with $\sigma_G(g)\,=\, {\widetilde D}_1{\overline g}
{\widetilde D}_1^{-1}$}. (We note that the inner and outer automorphisms here tend to
get confused, as the outer automorphisms are inner for the slightly larger group
${\rm O}(2n, {\mathbb C})$.) This case reproduces the previous one, in essence.

{\bf Table of real, pseudo-real structures over a point.}
\bigskip

{\tiny{\begin{tabular}{|l|c|c|c|c|}\hline
&&&&\\$G,\sigma_G(g)$&$H^1(\bbz/2\bbz, Z)$ &$H^1(\bbz/2\bbz, G), [H^1_c(\bbz/2\bbz, G)]$& $H^1(\bbz/2\bbz, G_{ad})$& $H^2(\bbz/2\bbz, Z)$  \\ \hline\hline &&&& \\ $\bbc^*,\ \overline g^{-1}$& $\{ \pm 1\}$&$\{ \pm 1\},\ $&$\{ 1\}$&$\{1\}$
\\ \hline &&&& \\ $\bbc^*,\ \overline g $& $\{ 1\}$&$\{1\} $&$\{1\} $&$\{\pm 1\} $
\\ \hline &&&& \\ ${\rm GL}(2),\  (g^*)^{-1} $& $\{\pm 1\} $&$\{\pm 1, iJ\simeq diag(1,-1)\}  $ &$\{ 1, J\} $&$\{1\} $\\ \hline &&&& \\ ${\rm GL}(n),n>2,\  (g^*)^{-1} $& $\{\pm 1\} $&$\{diag(1,\cdots,1,-1,\cdots,-1)\}  $ &$\{diag(1,\cdots,1,-1,\cdots,-1)\}/{\pm 1} $ &$\{1\} $
\\ \hline &&&& \\ ${\rm SL}(2n), n>1,\  (g^*)^{-1} $& $\{\pm 1\}$&$\{diag(1,\cdots,1,-1,\cdots,-1)\}  $ &$\{diag(1,\cdots,1,-1,\cdots,-1)\}/{\pm 1} $ &$\{\pm 1\} $\\&&(even number of $-1$)&&\\&&$[\{diag(i,\cdots,i,-i,\cdots,-i)\}_{c=-1}$&&\\&&(odd number of -i)]&&
\\ \hline &&&& \\ ${\rm SL}(2n+1), \  (g^*)^{-1} $& $\{  1\}$&$\{\pm diag(1,\cdots,1,-1,\cdots,-1)\}  $ &$\{diag(1,\cdots,1,-1,\cdots,-1)\}/{\pm 1} $ &$\{  1\} $\\&&(choose sign so that ~
$\det = 1$)&&
\\ \hline &&&& \\ ${\rm GL}(2n),n>1,\  \overline g $& $\{ 1\} $&$\{1\}, [\{J\}_{c=-1}] $ &$\{1,J\}$ &$\{\pm 1\} $
\\ \hline &&&& \\ ${\rm GL}(2n+1),\  \overline g $& $\{ 1\} $&$\{1\} $ &$ \{ 1\} $ &$\{\pm 1\} $
\\ \hline&&&& \\ ${\rm SO}(2n,\bbc) , n>1,\  \overline g$& $\{\pm 1\}$&$\{diag(1,\cdots,1,-1,\cdots,-1)\},  [\{J\}]  $ &$\{J,diag(1,\cdots,1,-1,\cdots,-1)\}/{\pm 1} $ &$\{\pm 1\} $\\&&(even number of $-1$)&(even number of $-1$)&
\\ \hline&&&& \\ ${\rm SO}(2n+1,\bbc) , \  \overline g$& $\{ 1\}$&$\{diag(1,\cdots,1,-1,\cdots,-1)\}  $ &$\{ diag(1,\cdots,1,-1,\cdots,-1)\}/{\pm 1} $ &$\{ 1\} $\\&&(even number of $-1$)&(even number of $-1$)&\\ \hline
\end{tabular}}}

\subsection{Real bundles over a circle}

Now consider a real bundle over a circle fixed by the real structure. Over the circle, 
there is then a fixed class in $H^1(\bbz/2\bbz,\, G)$. This class corresponds to the
restriction of the bundle to a point of the circle.
We can assume that the bundle is 
trivialized as a complex bundle. However, one has the real structure on the bundle 
defined by $\widetilde\sigma(g)\,=\, h(t)\cdot \sigma_G(g)$, where $t$ is a parameter 
along the circle. Since the real structures form a discrete set, now change 
trivializations along the circle, so that $h(t)\,=\, h(0)\,=\, h$. This of course can mean 
that going all the way round the circle ($t= 1$) one no longer has a trivialization of 
the bundle on the full circle; rather there is a holonomy $T\,\in\, Stab(h)$. Of course, if 
$T$ lies in the connected component of the identity of $Stab(h)$, one can then 
normalize to $T\,=\, 1$, and the bundle is trivial as a real bundle. More generally, the 
different real bundles corresponding to our element in $H^1(\bbz/2\bbz,\, G)$ are 
classified by $\pi_0(Stab(h))$, where $h$ represents our class in $H^1(\bbz/2\bbz, \,G)$.
Here the stabilizer is under the action $h\,\longmapsto\,\sigma_G(g)hg^{-1}$. The stabilizer 
is the real subgroup $ G_{\bbr,h}$ of elements invariant under the real structure 
$g\,\longmapsto\, h^{-1}\sigma_G(g)h$.

Now assume that there are several real circles. Fixing the real structure on the group, 
there is no guarantee that, having, for example, one type of real bundle structure over 
one circle forces it to be the same over the rest. To give just a simple example, let 
$H$ represent a non-trivial class in $H^1(\bbz/2\bbz,\, G)$; let
$$h\,:\,[0\, ,\pi]\,\longrightarrow\, G$$ 
be a path with $h(0)\,=\,\bbi$, $h(\pi) \,=\, H$, and extend this to $[-\pi\, , \pi]$ by 
$h(-\theta)\,=\, \sigma_G(h(\theta))^{-1}$. Now consider the torus parametrized by 
$\{(\theta,\psi)\in [0,\pi]\times[0,\pi]\}$, with real structure $(\theta\, ,\psi)\,
\longmapsto\, (-\theta\, ,\psi)$. Now consider the trivialized bundle $$\{(\theta\, ,\psi
\, , g)\,\in\, [0\, ,\pi]\times[0\, ,\pi]\times G\}$$ with real structure $(\theta\, ,\psi
\, ,g)\,\longmapsto\, 
(-\theta\, ,\psi\, ,\sigma_G(g)h(\theta))$; over the fixed locus $\theta\,=\,0$, it has fixed 
points, while over the fixed locus $\theta\,=\,\pi$, it does not. Thus, to each fixed 
circle, we should have a class in $H^1(\bbz/2\bbz,\, G)$, and then a class in 
$\pi_0(Stab(h))$.

{\it Example.} For $G\,= \,{\rm GL}(n,{\mathbb C})$ with the real structure given by
conjugation, one has two 
elements in $\pi_0(Stab(h))\,=\, \pi_0({\rm GL}(n,\bbr))$, and one obtains the Stiefel-Whitney 
class. For the real structure $a\,\longmapsto\, (a^*)^{-1}$, the groups $Stab(h)$ are simply 
the various groups $U(p,q)$, which are connected.

{\bf Table of components of $\pi_0(Stab (h)$:}

{\footnotesize{\begin{tabular}{|l|c|c|c|}\hline
&&& \\$G,\sigma_G(g)$&$h\in H^1(\bbz/2\bbz, G), [H^1_c(\bbz/2\bbz, G)]$& $ Stab (h)$&$\pi_0(Stab (h))$
   \\ \hline  &&&  \\ $\bbc^*,\ \overline g^{-1}$& $\{ \pm 1\},\ [\{ \pm i\}_{(c=-1)}]$&$S^1$&$\{1\}$
\\ \hline &&&  \\ $\bbc^*,\ \overline g $& $\{ 1\}$&$\bbr^*$&$\bbz/2\bbz$
\\ \hline &&&  \\ ${\rm GL}(2),\  (g^*)^{-1} $& $\{\pm 1\}$&$U(2)$&\{1\}
\\  &$iJ\simeq diag(1,-1)  $ &$U(1,1)$&$\{1\} $
\\ \hline &&&  \\ ${\rm GL}(n),n>2,\  (g^*)^{-1} $&  $ diag(1,\cdots,1,-1,\cdots,-1)   $ &$U(p,q) $ &$\{1\} $
\\ \hline &&&  \\ ${\rm SL}(2n), n>1,\  (g^*)^{-1} $& $\{diag(1,\cdots,1,-1,\cdots,-1)\}  $ &${\rm SU}(p,q)$ &$\{1\} $\\&(even number of $-1$)&&\\&$[\{diag(i,\cdots,i,-i,\cdots,-i)\}_{c=-1}$&&\\&(odd number of -i)]&&
\\ \hline &&&  \\ ${\rm SL}(2n+1), \  (g^*)^{-1} $&  $\{\pm diag(1,\cdots,1,-1,\cdots,-1)\}  $ &$
{\rm SU}(p,q)$ &$\{  1\} $\\&(choose sign so that ~$\det = 1$)&&
\\ \hline &&&  \\ ${\rm Gl}(2n),n>1,\  \overline g $& $\{1\}, [\{J\}_{c=-1}] $ &${\rm Gl}(2n,\bbr) [{\rm Gl}(n,\bbh)]$ &$\bbz/2\bbz,[\{  1\}]$
\\ \hline &&&  \\ ${\rm GL}(2n+1),\  \overline g $& $\{ 1\} $&${\rm GL}(2n+1,\bbr)$ &$\{ 1\} $
\\ \hline&&&  \\ ${\rm SO}(2n,\bbc) , n>1,\  \overline g$& $\pm\{diag(1,\cdots,1)\},  [\{J\}]  $ &${\rm SO}(2n), [SU^*(n)]$ &$\{ 1\} $\\ &$\{diag(1,\cdots,1,-1,\cdots,-1)\}$& ${\rm SO}(p,q)$&$\bbz/2\bbz$\\&$(p\neq 0\neq q $=number of $-1$ even) & &
\\ \hline&&&   \\ ${\rm SO}(2n+1,\bbc),\  \overline g$& $\pm\{diag(1,\cdots,1)\}  $ &${\rm SO}(2n), $ &$\{ 1\} $\\ &$\{diag(1,\cdots,1,-1,\cdots,-1)\}$& ${\rm SO}(p,q)$&$\bbz/2\bbz$\\&$(p\neq 0\neq q $=number of $-1$ even) & &\\ \hline
\end{tabular}}}

\subsection{Constructing a universal bundle}

We recall Milnor's construction of the classifying space as an infinite 
join \cite{Mi}. We take $n$ copies of $[0,1]\times G$, and form the 
space
$$
EG^n\,=\, \{((t_1,g_1),\cdots ,(t_n, g_n))\in ([0,1]\times G)^n| 
\sum_i t_i 
= 1\}/\equiv
$$
Here the equivalence relation is given by identifying $(0, g_i)$ with 
$(0,g'_i)$. This space has a free right action of $G$, given by right 
multiplication on each factor: $(t_i\, , g_i)\,\longmapsto \,(t_i\, , g_ig)$. We 
denote the quotient by $BG^n$, and $EG^n$ is a principal $G$-bundle 
over $BG^n$. The space is $n-2$ connected, and taking an appropriate 
limit in $n$ gives the classifying space.

We now put in a real or pseudo-real structure. Choosing an $h$ representing a class 
$\alpha\,\in\, H^1_c(\bbz/2\bbz,\, G)$, one then simply acts on all the $g_i$ by 
$g_i\,\longmapsto\, h\sigma_G(g_i) $. This then automatically satisfies $ 
g_ig\,\longmapsto\, h\sigma_G(g_i)\sigma_G(g )$, and so orbits are mapped to orbits, and
the involution descends to $BG$. If the class $\alpha$ is trivial (this forces $c\,=\, e$) the
fixed point set of the involution on $EG$ is then $EG_\bbr\,\longrightarrow\, BG_\bbr$. If
$\alpha$ is not trivial, the involution on $EG$ has no fixed points, but on the base, for a
fixed point one has the condition that $a_i\,=\, \sigma_G(g_i)hg_i^{-1}$ be the same for
all $i$ for which $t_i\,\neq\, 0$; one can normalize $a_i$ to $h$, and so one is restricting
to the set $g_i\,\in\, Stab(h)\,= \,G_{\bbr,h}$, the group of real points for the real structure 
$h^{-1}\sigma_Gh$. The fixed point set is then $BG_{\bbr,h}$. We use the same notation for 
the pseudo-real structures, i.e., $Stab(h)\,=\, G_{\bbr,h}$.

We will now build a copy of the classifying space which has all of these real or pseudo-real
structures at once. We take representative elements $h_{\alpha_j}$ for each of the $n$
elements $\alpha_j\, , j\,=\,1,\cdots,n$ of $H^1_c(\bbz/2\bbz,\, G)$. Now define the action by 
\begin{align*}
(\cdots ,(t_{k(n+2)+1}, g_{k(n+2)+1}),&(t_{k(n+2)+2}, g_{k(n+2)+2}),\cdots,(t_{k(n+2)+n}, g_{k(n+2)+n})\cdots)
\end{align*}
\begin{align*}
\longmapsto\,
(\cdots(t_{k(n+2)+1}, h_{\alpha_1}\sigma_G(g_{k(n+2)+1})),&(t_{k(n+2)+2},h_{\alpha_2}
\sigma_G(g_{k(n+2)+2}),
\end{align*}
\begin{align*}
\cdots,(t_{k(n+2)+n}, h_{\alpha_n}\sigma_G(g_{k(n+2)+n}),\cdots)\, .
\end{align*}

One has, in the fixed point set on the quotient $BG$, a disjoint union of the 
different $BG_{\bbr,h_{\alpha_j}}$, each with a base point $(t_i\, , g_i)\,=\, 
(\delta_{i,j}\, , e)$. Above each component, we have a reduction of the structure group 
to $G_{\bbr,h_{\alpha_j}}$.

Now suppose that we have a Riemann surface $X$ with an antiholomorphic
involution $\sigma_X$, with fixed curves $\{\gamma_i\}$ 
for $\sigma_X$. As mentioned in the first section (more details 
are given in \cite{BHH}), one can write $X$ as a union of $X_0$ and $\sigma_X(X_0)$, 
where $X_0$ and $\sigma_X(X_0)$ are surfaces with boundary and intersect along the 
$\{\gamma_i\}$ and possibly one or two other boundary curves $\delta_1\, ,\delta_2$, which 
are mapped to themselves by $\sigma_X$. One can build a cell decomposition of $X_0$ 
composed of a set of 0-cells, 1-cells that are either fixed (and so part of a 
$\gamma_i$), or lie in the $\delta_j$, or whose interiors do not intersect the 
$\gamma_i$ (these last ones span some of the homology cycles of $X$), and a two 2-cell 
$e$ whose interior is disjoint from $\sigma_G(e)$. Let the portion of the 
one-skeleton that is away from the boundary of $X_0$ be denoted by $S$. Choose for 
each $\gamma_i$ a class $\alpha_i\,\in\, H^1(\bbz/2\bbz, \,G)$; bundles over the $\gamma_i$ 
corresponding to the class $\alpha_i$ are classified by an element $\beta_i$ of the 
fundamental group $\pi_1(BG_{\bbr,\, h_{\alpha_i}}) \,= \,
\pi_0(G_{\bbr,\, h_{\alpha_i}})$. Now let us try to fill in the map to the rest of $X_0$.

We begin with the boundary curves $\delta_i$. The $G$-bundle is trivial over these 
components; the curves must however be mapped $\sigma_G$ invariantly into $BG$. There is, 
up to homotopy, one way of doing this, using the fact that $\pi_0(BG)\,=\,\pi_1(BG)\,=\,0$. In
the same way, the remainder of the 1-skeleton maps uniquely up to homotopy into $BG$. Once 
this is fixed, what is left is a map of the disk into $BG$, determined up to an element of 
$\pi_2(BG)$. In short, for the relative homotopy $[(X_0\, , \sqcup_i \gamma_i)\, , (BG, 
\sqcup_i BG_{\bbr\, ,h_{\alpha_i}})]$, one has a fibration of sets, describing the
equivalence classes of real bundles:
$$\pi_2(BG)\,=, \pi_1(G)\,\longrightarrow\,
[(X_0\, , \sqcup_i \gamma_i)\, , (BG\, , 
\sqcup_i BG_{\bbr\, ,h_{\alpha_i}}])
$$
$$
\longrightarrow \, \prod_i \pi_1(BG_{\bbr\, ,h_{\alpha_i}})
\,= \,\prod_i \pi_0(G_{\bbr\, ,h_{\alpha_i}})\, .$$
As we have restricted to $X_0$, the requirement of equivariance has almost disappeared; 
there is one residual constraint of equivariance, in that the boundary curves $\delta_i$ 
have to be mapped equivariantly; we have seen however that this is trivial.
 
While the general question of building and classifying equivariant bundles is intricate, 
in this simple case of curves, one has:

\begin{theorem}
The topological classes of pseudo-real bundles on $X$ for a fixed $c\,\in \,H^2(\bbz/2\bbz,\,Z)$
are determined by equivariant homotopy classes of equivariant maps of $X$ into the classifying
space. This amounts to the following:
\begin{itemize}
\item The choice, for each fixed curve $\gamma_i$, of a class $\alpha_i\, ,i\,=\,1,\cdots
,r$ of $H^1_c(\bbz/2\bbz, \,G)$; this is the topological type of the bundle over each point
of $\gamma_i$.

\item The choice, for each fixed curve $\gamma_i$, of a generalized Stiefel Whitney class 
$\beta_i$ in $\pi_1(BG_{\bbr,h_{\alpha_i}})\,=\, \pi_0(G_{\bbr,h_{\alpha_i}})$.

\item The choice of a relative class $\rho$ in $[(X_0\, , \sqcup_i \gamma_i)\, 
 (BG\, , \sqcup_i BG_{\bbr,h_{\alpha_i}})]$, mapping the boundary curves $\gamma_i$ to their corresponding Stiefel Whitney classes, and the curves $\delta_i$ equivariantly. 
\end{itemize}
\end{theorem}
 We note that the class $\rho$ gets ``doubled" when one extends equivariantly to $X = X_0\cup \sigma_X(X_0)$; this results in a parity constraint on the degree of the bundle over all of $X$, as in \cite{BHH}.
 
One can give a fairly explicit description of this in terms of the space $EG^2$, 
and its quotient $BG^2$. It is not hard to see that the latter space is 
homeomorphic to $\Sigma(G)$, the unreduced suspension of $G$, obtained by taking 
the product of $G$ with the unit interval, and collapsing $\{0\}\times G$ to a 
point, and $\{1\}\times G$ to a point. One can obtain the maps we want into $BG$ as 
maps to $BG^2$. In general, we are not able to make our copy of $BG^2$ equivariant 
in a suitable way; nevertheless, one is able to map the surface $X_0$ into the copy 
of $BG^2$ in a way which gives the relationship between our invariants associated 
to the class $h$ and the Stiefel-Whitney classes to the overall degree of the 
bundle on $X$.

Indeed, as we have seen, we can contract cycles on $X_0$ so that the result 
$\widetilde X_0$ is a sphere punctured along disks, with boundary circles 
$\gamma_i$ and possibly $\delta_j$; we can consider maps from $\widetilde X_0$ 
instead, as what we have contracted is homotopically trivial in $BG$. Now one can 
write $\widetilde X_0$ up to homotopy as the unreduced suspension $\Sigma_C$ of 
$C$, a circle punctured along intervals, with the boundary of each circle being a 
pair of points on one of the $\gamma_i$, or $\delta_j$. We choose the points on 
$\delta_j$ so that they are preserved under the involution. One can then give a map 
$\widetilde X_0\,=\, \Sigma C\,\longrightarrow\, \Sigma G$ in terms of a map 
$C\,\longrightarrow\,G$. We can even restrict somewhat, so that the map of each 
left hand boundary point of the gaps in $C$ is mapped to the identity. For the 
boundary points corresponding to the $\gamma_i$, map the right boundary point to a 
representative element of the Stiefel-Whitney class in 
$\pi_0(G_{\bbr,h_{\alpha_i}})$, For $\delta_j$, map the right boundary point to the 
element $c$. Now one chooses an element $\alpha$ of $Map (C, G)$ extending the maps 
chosen of the boundary points. Consider the union $D$ of two copies $C_1, C_2$ of 
$C$ along their boundaries; this is a family of circles. Extending the map from $C$ 
to $D$ by $\alpha, \sigma_G\circ \alpha$ then gives a map from a family of circles 
into $G$, representing a class in $H^1(G,\bbz)\,=\, \pi_1(G)$. This will be the 
characteristic class of the bundle one has on $X$, and one can deduce the 
possibilities for this cycle from the choices made; in particular, one can recover 
the constraints of \cite{BHH} for real vector bundles that $c_1(E)$ mod 2 is the 
sum of the Stiefel Whitney classes along the real components $\gamma_i $, and that 
for quaternionic vector bundles, that $c_1(E)$ is equal to the product of the rank 
and the genus minus one, mod two.

\section{Stable pseudo-real principal bundles}

\subsection{Stable and semistable principal bundles}

We now want to consider semistable and stable holomorphic real bundles. we recall 
some results from \cite{BH}.

Let $(E \, ,\sigma_E)$ be a pseudo-real principal $G$--bundle
over $X$. Let
$$
\text{Ad}(E)\, :=\, E \times^G G\, \longrightarrow\, X
$$
be the group-scheme over $X$ associated to $E $ for the adjoint
action of $G$ on itself. As in \cite{BH}, the involution $\sigma_E$ induces 
\begin{equation}\label{e5}
\begin{matrix}
\text{Ad}(E ) &\stackrel{\sigma_{\text{Ad}}}{\longrightarrow}&
\text{Ad}(E )\\
\Big\downarrow && \Big\downarrow\\
X&\stackrel{\sigma_X}{\longrightarrow}& X
\end{matrix}
\end{equation}

Let
$$
\text{ad}(E )\, :=\, E \times^G{\mathfrak g}
\, \longrightarrow\, X
$$
be the bundle of Lie algebras over $X$ associated to $E $ for the 
adjoint action of $G$ on ${\mathfrak g}\, :=\, \text{Lie}(G)$;
it is called the \textit{adjoint} vector bundle. Again, the involution $\sigma_E$ induces
an antiholomorphic involution $\sigma_{ad}$ of this bundle $\text{ad}(E)$.

A \textit{proper parabolic} subgroup-scheme of $\text{Ad}(E )$ is
a Zariski closed analytically locally trivial proper
subgroup-scheme $\underline{P}\,\subset\,
\text{Ad}(E_G)$ such that $\text{Ad}(E_G)/\underline{P}$ is compact. For
an analytically locally trivial
subgroup-scheme $\underline{P}\,\subset\, \text{Ad}(E_G)$, let
$\underline{\mathfrak p}\, \subset\, \text{ad}(E_G)$ be the bundle of
Lie subalgebras corresponding to $\underline{P}$.

Let $\underline{P}\,\subset\, \text{Ad}(E_G)$ be a proper parabolic subgroup-scheme.
For each point $x\, \in\, X$, the unipotent radical of the fiber
$\underline{P}_x$ will be denoted by $R_u(\underline{P})_x$. We recall that
$R_u(\underline{P})_x$ is the unique
maximal normal unipotent subgroup of $P_x$. We have a holomorphically
locally trivial subgroup-scheme
$$
R_u(\underline{P})\, \subset\, \underline{P}
$$
whose fiber over any $x\, \in\, X$ is $R_u(\underline{P})_x$. The quotient
$\underline{P}/R_u(\underline{P})$ is a group-scheme over $X$.

A \textit{Levi subgroup-scheme} of $\underline{P}$ is an analytically
locally trivial subgroup-scheme
$L(\underline{P})\, \subset\, \underline{P}$ such that the composition
$$
L(\underline{P})\, \hookrightarrow\, \underline{P}\, \longrightarrow\, \underline{P}/R_u(\underline{P})
$$
is an isomorphism. It should be emphasized that a Levi subgroup-scheme
does not exist in general. In vector bundle terms, the existence of a 
Levi subgroup-scheme corresponds to some extension classes being trivial.

\begin{definition}\label{def3}
A pseudo-real principal $G$--bundle $(E \, ,\sigma_E)$
over $X$ is called \textit{semistable} (respectively, \textit{stable})
if for every proper parabolic subgroup-scheme $\underline{P}\,\subset\, 
\text{Ad}(E )$ invariant under $\sigma_{\text{Ad}}$, meaning $\sigma_{\text{Ad}}(\underline{P})\, \subset\, \underline{P}$,
the inequality
$$
{\rm degree}(\underline{{\mathfrak p}})\, \leq\, 0~\, \ {\rm (respectively,~}\,
{\rm degree}(\underline{{\mathfrak p}})\, <\, 0{\rm )}
$$
holds.
\end{definition}

\begin{definition}\label{def4}
A semistable pseudo-real principal $G$--bundle $(E \, ,\sigma_E)$
over $X$ is called \textit{polystable} if either $(E \, ,\sigma_E)$ is stable,
or there is a proper Levi subgroup-scheme $L(\underline{P})\,\subset\, \text{Ad}(E)$,
such that the following conditions hold:
\begin{enumerate}
\item $\sigma_{\text{Ad}}(\underline{P})\, \subset\, \underline{P}$, and 
$\sigma_{\text{Ad}}(L(\underline{P}))\, \subset\, L(\underline{P})$, and

\item for any proper parabolic subgroup-scheme $\underline{P}'\, \subset\, L(\underline{P})$
with $\sigma_{\text{Ad}}(\underline{P}')\, \subset\, \underline{P}'$, we have
$$
{\rm degree}(\underline{{\mathfrak p}}')\, <\, 0\, ,
$$
where $\underline{{\mathfrak p}}'$ is the bundle of Lie algebras corresponding
to $\underline{P}'$.
\end{enumerate}
\end{definition}

The above definition of (semi)stability coincides with the one in \cite[page 304,
Definition 8.1]{Be}, but it differs from the definition
of (semi)stability given in \cite{Ra}. The definitions of Behrend
and Ramanathan are equivalent if the base field is
$\mathbb C$. Over $\bbr$, Behrend's definition works better.
See \cite{BH} for a discussion. From \cite{BH} one also has:
 
\begin{proposition}[\cite{BH}]\label{prop1}
A pseudo-real principal $G$--bundle $(E \, ,\sigma_E)$ over $X$ is
semistable (respectively, stable) if the principal 
$G$--bundle $E $ is semistable (respectively, stable).

For a semistable pseudo-real principal $G$--bundle $(E \, ,\sigma_E)$, 
the principal $G$--bundle $E$ is semistable.

A pseudo-real principal principal $G$--bundle $(E \, ,\sigma_E)$ is polystable
if and only if the principal $G$--bundle $E $ is polystable.
\end{proposition}

The analog of Proposition \ref{prop1} for stable bundles is not true; see \cite{BH}.

\section{Gauge theory}
 
For any principal bundle, one can consider the affine space of connections on it, 
and this, following Atiyah and Bott (\cite{AtBo}), has been an extraordinarily 
effective tool for studying stable holomorphic bundles. The general idea, promoted 
by Atiyah and Bott, and put on a firmer Morse-theoretical footing by Daskalopoulos 
\cite{Da}, is that one considers the affine space of hermitian connections on a 
fixed topological bundle, and then quotients this by the action of the gauge group. 
This is very close to a classifying space for the group, and so one can compute its 
cohomology. In the meantime, one has the $L^2$ norm of the curvature, which 
provides a Morse function on this quotient space. Atiyah and Bott then show that 
this function is equivariantly perfect, so that one can obtain the cohomology of 
the minimal energy stratum (by the Narasimhan-Seshadri theorem, this is the moduli 
of semi-stable bundles) in terms of the cohomology of the whole, ``minus'' the 
cohomology of the higher critical sets, which are computable in terms of moduli 
spaces of bundles of lower rank, allowing an inductive process. In particular, from 
the Morse theory, one sees fairly immediately that the moduli space of bundles is 
connected.

One can look at the action under pull back by the real structure on all these spaces, and obtain the various real moduli as components of the fixed point set on the complex moduli. Alternately, however, one can proceed as in \cite{BHH}, fix the real topological structure of the bundle, and build the invariance directly into the space of connections and into the gauge group. This will allow us to see that the spaces of bundles are connected, for each topological type.

Indeed, let us suppose chosen a $\sigma_X$ invariant K\"ahler metric on $X$. Let us fix a $\sigma_E$-invariant reduction $E_K$ of $E$ to our $\sigma_G$-invariant maximal compact group $K$ (in vector bundle terms, this would be an invariant metric; the existence of such reductions is guaranteed by the fact that $G/K$ is contractible). Let $\mathcal A$ be the affine space of $K$-connections on the bundle $E_K$, and let ${\mathcal A}^\sigma$ be the connections invariant under the involution or pseudo-involution $\sigma^E$ on $E_K$. This is also an affine space. It is acted on by the group of $\sigma_K$-equivariant automorphisms $\mathcal K^\sigma$ of $E_K$. 

The space ${\mathcal A}^\sigma/\mathcal K^\sigma $ of invariant connections can be 
described in a way similar to that given in \cite{BHH}. The point is that the 
action of $\mathcal K^\sigma$ is almost free; stabilizers are finite dimensional, 
and generically just the real centers of the group. This allows us to consider
the classifying space $B{\mathcal K}^\sigma$ 
instead of the quotient ${\mathcal A}^\sigma/\mathcal K^\sigma $, and then to 
``adjust''. (For computations involving cohomology, this adjustment involves using equivariant cohomology throughout, and then relating the final result to ordinary cohomology; for homotopical issues, the adjustment depends more on the precise situation, and basically involves dealing with the covering corresponding to the center of the group.) 

The space $\mathcal K^\sigma$, in turn, is a space of invariant sections 
of the automorphism bundle of $E$; along subspaces over which the bundle is 
trivial, this is a mapping space into $K$. The main tool for treating the gauge 
group, and so the classifying space is the description of the surface $X$ as a 
union $X_0\cup \sigma(X_0)$; the union is taken along the boundary. Essentially, if 
one has $\sigma$-invariance along the boundary, one can then complete a map from 
$X_0$ to a $\sigma$-invariant map on $X$. If $\mathcal K(A)$ (respectively, 
$\mathcal K(A)^\sigma$) refers to automorphisms (respectively, invariant 
automorphisms) of the bundle along $A$, one has a diagram, where the column on the 
left is a pull back of the column on the right:
$$\begin{matrix} 
\mathcal K^\sigma(X) & \longrightarrow &\mathcal K (X_0)\\
\Big\downarrow && \Big\downarrow\\
\mathcal K^\sigma(\partial X_0)&\longrightarrow & \mathcal K (\partial X_0).
\end{matrix}$$
This is essentially the approach adopted in \cite{BHH}; one can use this to compute the first few homotopy groups of $\mathcal K^\sigma(X) $; the answers are already fairly intricate for ${\rm GL}(n,{\mathbb C})$ with the
standard conjugation real structure, and we will not pursue this here. We note that the structure of $\mathcal K^\sigma(\partial X_0)$ depends quite crucially on the
topological structure $t$ one has fixed.

One has the $L^2$ norm of the curvature functional
$$ F\,:\, {\mathcal A}^\sigma/\mathcal K^\sigma \,\longrightarrow\, \bbr$$
and the theorem of Atiyah and Bott extends to our equivariant context:

\begin{theorem}
The minima of the functional $F$ are attained at $\sigma_E$-invariant connections with 
constant central curvature. The space of minima is homeomorphic to the moduli space 
${\mathcal M}^\sigma_t$ of $\sigma_G$-semistable principal $G$-bundles of the given topological 
type $t$.
\end{theorem}

In the absence of $\sigma_E$, this is the theorem of Donaldson and Ramanathan \cite{ Do,Ra}. One has a
principle of symmetric criticality: minimizing energy  from an equivariant initial configuration gives an
equivariant result: the direction of the gradient flow is unique, and so must invariant under the involution. To apply the results of  Donaldson and Ramanathan, one needs the relationship between real semistability and ordinary semistability discussed above.

The space ${\mathcal A}^\sigma/\mathcal K^\sigma$ is connected. Let us now consider
higher order critical
points of the energy. These occur when there is a destabilizing parabolic $\sigma_{\text{Ad}}$-invariant
subgroup-scheme $\underline{P}$ of $\text{Ad}(E)$, that is 
$$
{\rm degree}({\underline{\mathfrak p}})\, >\, 0\, .
$$
The criticality in fact means that there is a proper $\sigma_{\text{Ad}}$-invariant 
Levi subgroup-scheme $L(\underline{P})$ of $\underline{P}$; on the level of Lie 
algebras one has $\underline{{\mathfrak l}}\subset \underline{{\mathfrak p}}$. The 
index of the critical point is given by the dimension of the real subspace of 
$H^1(X,\, (\underline{{\mathfrak p}}/\underline{{\mathfrak l}})^\vee)$. One has 
that the degree $d$ of $(\underline{{\mathfrak p}}/\underline{{\mathfrak 
l}})^\vee)$ is less than zero, since we have a destabilizing bundle. Riemann-Roch 
tells us that the dimension of $H^1(X, (\underline{{\mathfrak 
p}}/\underline{{\mathfrak l}})^\vee)$ is at least $-d+ k (g-1)$, where $k$ is the 
smallest codimension of a parabolic subgroup of $G$. Thus:

\begin{theorem}
If $g$ is at least two, the moduli space ${\mathcal M}^\sigma_t$ is 
connected.
\end{theorem}

Indeed, our estimate tells us that the index of all the critical points on our 
space of connections is at least two, so that any path in ${\mathcal 
A}^\sigma/\mathcal K^\sigma$ joining two points of ${\mathcal M}^\sigma_t$ can be 
pushed down into ${\mathcal M}^\sigma_t$.

\section{Real components in the complex moduli space}

We have seen that once one fixes the topological type of a real bundle $(E,\sigma_E)$, one obtains a connected moduli space. Alternately, one can look at the effect of the real involution on $X$ on the complex moduli space ${\mathcal M}$. For this let 
$$\overline E \,=\, E\times ^{\sigma_G} G\, .$$
The involution is given by 
$$I_\sigma: E\,\longmapsto \,\sigma_X^* (\overline E)\, .$$
 The fixed points occur when there exists a lift of $\sigma_X$ to the bundle; the lift is 
not specified. If at a fixed point the bundle in question is regularly stable, i.e., is 
stable and has no automorphisms apart from those given by the center of the group, then 
one has a unique real or pseudo-real structure, up to automorphism; indeed, if there are 
two of them, then composing, the bundle has a non-central automorphism, contradicting our 
hypothesis. We recall from \cite {BHo}:
 
\begin{proposition}[{\cite[Proposition 2.3]{BHo}}] The smooth locus of the moduli space ${\mathcal M}$ of
$G$--bundles over $X$ consists of regularly stable bundles, as long as
$g\,\neq \,2$ or ${\rm PSL}(2,{\mathbb C})$ is not a factor of $G/Z$.\end{proposition}
 
Now let us restrict to the fixed point locus   ${\mathcal M}_{\sigma_X}$ of $I_\sigma $;  each component of the smooth locus of the moduli of stable real (or pseudo-real) 
bundles has associated to it a unique topological type of real (or pseudo-real) structure. One then has a 
lower bound:

\begin{proposition} Let $g\,> \,2$, and suppose  ${\rm PSL}(2,{\mathbb C})$ is not a factor of $G/Z$. The number of components of the smooth locus of ${\mathcal M}_{\sigma_X}$ is bounded below  by the number of possible topological types. 
\end{proposition}

For ${\rm GL}(n,
{\mathbb C})$, we 
note that regularly stable and stable are the same; furthermore, if the degree and rank 
are coprime, then stability and semistability coincide, and our lower bound becomes a 
count.

{\it Example: ${\rm GL}(n,{\mathbb C}), \sigma_G(g) = \overline{g}$}:
For example, for vector bundles of rank $n$, and the standard real structure, for type I curves with r real curves, one has $2^r$ possibilities for the Stiefel-Whitney classes, and one extra possibility, the quaternionic bundles, in even rank. There is a constraint for real bundles that the sum of the Stiefel-Whitney classes must equal the degree, mod 2. For quaternionic bundles, the degree must be even. We find as a number of components (see Schaffhauser \cite{Sch}: 
\begin{itemize}
\item {\it r odd, rank odd, degree odd}: $\begin{pmatrix}r\\ 1\end{pmatrix} + \begin{pmatrix}r\\ 3\end{pmatrix} +\ldots + \begin{pmatrix}r\\ r\end{pmatrix} = 2^{r-1}$

\item {\it r odd, rank odd, degree even}: $\begin{pmatrix}r\\ 0\end{pmatrix} + \begin{pmatrix}r\\ 2\end{pmatrix} +\ldots+ \begin{pmatrix}r\\ r-1\end{pmatrix}= 2^{r-1}$
\item {\it r odd, rank even, degree odd}: $2^{r-1}$.
\item {\it r odd, rank even, degree even}: $2^{r-1}+1$ (There is the additional quaternionic moduli space)
\end{itemize}
with similar results for $r$ even; one can also perform a similar analysis for type 0 or II curves.

{\it Example: ${\rm GL}(n,{\mathbb C}), \sigma_G(g) = (g^*)^{-1}$}: Again on type I curves, 
for principal ${\rm GL}(n,{\mathbb C})$ bundles with the real structures $\sigma_G(g)\,=\,(g^*)^{-1}$, we find 
that the topological types over each real component are defined by a signature lying in 
the set $0,\cdots,n$. This implies that there are $r^{n+1}$ possibilities over the real 
curves, each giving a component. The bundles must all be of even degree.

One of the main results of the theory, for $G$ semi-simple, is that polystable bundles 
correspond to flat connections and so to representations into the fundamental group. One 
can ask what representations are given by bundles left invariant by the real involution. 
The answer is given in \cite{BH}; we recall it briefly here.

 Fix
a base point $x\, \in\, X$ such that $\sigma(x)\, \not=\, x$. Let
$\Gamma$ be the space of all homotopy classes of paths on $X$
starting from $x$ and ending in either $x$ or $\sigma_X(x)$.
This set $\Gamma$ has a natural structure of a group, using the involution on $X$ to compose paths if necessary. In
turn, let $\widehat K$ be the $\bbz/2\bbz$ extension built from the automorphism $\sigma_G$.
Let ${\rm Hom}'(\Gamma\, ,\, \widehat{K})$ be the space of
all homomorphisms $\varphi\, :\, \Gamma\, \longrightarrow\,
\widehat{K}$ that fit in the commutative diagram
$$
\begin{matrix}
0 & \longrightarrow & \pi_1(X,x)& \longrightarrow & \Gamma &
\longrightarrow & {\mathbb Z}/2{\mathbb Z} &
\longrightarrow & 0\\
&& \Big\downarrow && ~\, \Big\downarrow \varphi && \Vert \\
0 & \longrightarrow & K & \longrightarrow & \widehat{K} &
\longrightarrow & {\mathbb Z}/2{\mathbb Z} &
\longrightarrow & 0
\end{matrix}
$$

One shows that real (invariant under the involution) bundles give rise to 
representations fitting into this diagram, and vice versa. The pseudo-real bundles 
give rise to ``twisted'' representations; see \cite{BH} for details.

Let us consider real bundles. We have now an understanding of the different 
components of the representation space; these are in essence given in terms of (1) 
a choice of classes $\alpha_i\in H^1(\bbz/2\bbz,\, G)$ represented by 
$h_{\alpha_i}$ for each real component $\gamma_i$ of $X$; (2) a choice of a 
generalized Stiefel-Whitney class $w(\gamma_i) \in \pi_0(Stab(h_i))$ for each 
$C_i$; (3) a choice of degree for the bundle, which is usually determined modulo 2 
and here must vanish as the group $G$ is semi-simple. One can ask what 
characterizes representations for these components.

We note that $h_{\alpha_i}$ can be chosen to lie in $K$, and are either the 
identity or elements of order two. The constraint on connections and so on 
representations imposed by the component is then fairly immediate: one has that the 
holonomy along the curves $\gamma_i$ must lie in the group $Stab(h_{\alpha_i})\cap 
K$, and so if one considers curves $\widetilde\gamma_i \,=\,p^{-1}\circ\gamma_i\circ p $, 
where $p$ is a path from the base point to a point of $\gamma_i$, these must map to 
a conjugate of the stabilizer. Likewise, the Stiefel-Whitney class is determined by 
the connected component of $Stab(h_{\alpha_i})\cap K$ in which the holonomy lies, 
after conjugation. A similar analysis is possible in the pseudo-real case.

\section*{Acknowledgements}

We thank the referee for detailed comments. The first-named author thanks ICMAT for 
hospitality where a part of the work was carried out. He also acknowledges the 
support of a J. C. Bose Fellowship.

%%%%%%%%%%%%%%%%%%%%%%%%%%%%%%%%%%%%%%%%%%%%%%%%%%%%%%%%%%% %%%%%%

\end{document}